\documentclass[twoside,11pt]{amsart}
\usepackage{latexsym}
\usepackage{amssymb}
\usepackage{epsfig}
\usepackage{multirow}
\usepackage{caption}
\usepackage{subfigure}
\usepackage{graphicx}
\usepackage{color,verbatim}
\usepackage[normalem]{ulem}

\newtheorem{thm}{Theorem}[section]

\newtheorem{lem}[thm]{Lemma}

\theoremstyle{definition}

\theoremstyle{remark}
\newtheorem{rem}{Remark}[section]

\makeatletter
\def\serieslogo@{}
\makeatother \makeatletter
\def\@setcopyright{}
\makeatother

\begin{document}
\title[Entropy satisfying DG methods]{An entropy satisfying discontinuous Galerkin method for nonlinear Fokker-Planck equations}

\author[H.~Liu and Z.~ Wang]{Hailiang Liu$^\dagger$ and Zhongming Wang$^\ddagger$}
\address{$^\dagger$Iowa State University, Mathematics Department, Ames, IA 50011} \email{hliu@iastate.edu}
\address{$^\ddagger$ Florida International University,  Department of Mathematics and Statistics,  Miami, FL 33199}
\email{zwang6@fiu.edu}
\begin{abstract}
We propose a high order  discontinuous Galerkin (DG) method for solving nonlinear Fokker-Planck equations with a gradient flow structure.
For some of these models  it is known that the transient solutions converge to steady-states when time tends to infinity. The scheme  is shown to satisfy a discrete version of the entropy dissipation law and preserve steady-states, therefore providing numerical solutions  with satisfying long-time behavior. The positivity of numerical solutions is enforced through a reconstruction algorithm, based on positive cell averages. For the model with trivial potential, a parameter range sufficient for positivity preservation is rigorously established. For other cases, cell averages can be made positive at each time step by tuning the numerical flux parameters.   A selected set of numerical examples is presented to confirm both the high-order accuracy and the efficiency to capture the large-time asymptotic. 
\end{abstract}

\maketitle

\section{Introduction}
In this paper,  we propose a high order accurate discontinuous Galerkin (DG) method for solving the following problem 
\begin{subequations}\label{fp}
\begin{align}
\partial_t u & =\nabla_x\cdot (f(u)\nabla_x (\Phi(x)+H'(u)) ),\quad x\in \Omega,\; t >0, \\
u(x, 0) & = u_0(x),
\end{align}
\end{subequations}
subject to appropriate boundary conditions. Here $ u(t, x)\geq 0$  is the unknown,  $\Omega$ is a bounded domain in  $\mathbb{R}^d$,
$H: \mathbb{R}^+ \to \mathbb{R}$  and $f: \mathbb{R}^+ \to \mathbb{R}^+$ are given functions, and $\Phi(x)$ is a given potential function.

 This equation has a gradient flow structure corresponding to the entropy functional
$$E=\int_{\Omega} (H(u) +u\Phi(x))dx.$$
A simple calculation shows that the time derivative of this entropy along the equation (\ref{fp}a) with zero flux boundary condition is
\begin{equation}\label{en}
\frac{d}{dt} E(t) =-\int_{\Omega} f(u)|\nabla_x(\Phi+H'(u))|^2dx \leq 0,
\end{equation}
which reveals the entropy dissipation property of the underlying system.  Certain entropy dissipation inequalities are  recognized to characterize the fine details of the convergence to steady states, see e.g.,  \cite{CJMTU,CMV03,CT, Otto}.

Equations such as (\ref{fp}a)  appear in a wide range of applications. In the case $f(u)=u$,  the equation becomes
\begin{equation}\label{fplinear}
\partial_t u =\nabla_x\cdot (u\nabla_x (\Phi(x)+H'(u))).
\end{equation}
If  $H'(u)=u^m (m>1)$ and $\Phi=0$,  it is the porous medium equation  \cite{CT,Otto}, and for $H'(u)=\nu u^{m-1}$ and $\Phi={x^4}/{4}-{x^2}/{2}$, it is the nonlinear diffusion equation confined by a double-well potential  \cite{CCH}.
A particular example with nonlinear $f(u)$  is
\begin{align}
\partial_t u=\nabla_x\cdot (x u(1+ku)+\nabla_x u), \label{GFP1}
\end{align}
which is known as a model for fermion ($k=-1$) and boson ($k=1$) gases \cite{CLR,CRS,Toscani}.
A more general class of the form
\begin{align}
\partial_t u=\nabla_x\cdot (x u(1+u^N)+\nabla_x u),\label{GFP2}\quad N>2,
\end{align}
is known to develop finite time concentration beyond some critical mass  \cite{AGT}.

In order to capture the rich dynamics of solutions to (\ref{fp}), it is highly desirable to develop high order schemes which can preserve the entropy dissipation law (\ref{en}) at the discrete level.  In this work, we propose such a scheme for \eqref{fp} using the discontinuous Galerkin discretization.

A related finite volume method was already proposed in \cite{BF} for  (\ref{fp}), 
and further generalized to cover the nonlocal terms and general dimension in \cite{CCH}. For  (\ref{fp}) with $f(u)=u$ and an additional nonlocal interaction term,  a mixed finite element method was studied in \cite{BCW} based on their interpretation as gradient flows in optimal transportation metrics, following the so called JKO formulation,  which is a variational scheme  proposed by Jordan, Kinderlehrer and Otto \cite{JKO98}  for linear Fokker-Planck equations.  Regarding the use of relative entropy functionals  we refer to \cite{AU} for the study of the large time behavior of a fully implicit semi-discretization applied to  linear parabolic Fokker-Planck type equations in the form of (\ref{fp}) with $f(u)=u$, $H=u {\rm log} u$.  A free energy satisfying finite difference method was proposed in \cite{LW} for the Poisson-Nernst-Planck (PNP) equations, which correspond to  (\ref{fp}) with $f=u$, $H=u{\rm log}u$, further coupled with a Poisson equation for governing the potential $\Phi$. However, these existing schemes are only up to second-order.

An entropy satisfying DG method has been recently developed in \cite{LY14a} for the linear Fokker-Planck equation
\begin{equation}\label{fplinear+}
\partial_t u =\nabla_x\cdot (\nabla_x u  + u \nabla_x \Phi),
\end{equation}
which corresponds to (\ref{fplinear}) with $H=u log u$. The obtained DG method generalizes and improves upon the finite volume method introduced in \cite{LY12}.  The idea in \cite{LY14a} is to apply the DG discretization to the non-logarithmic Landau formulation of (\ref{fplinear+}), 
$$
\partial_t u =\nabla_x\cdot \left(M\nabla_x \left(\frac{u}{M}\right)\right), \quad M=e^{-\Phi(x)},
$$
so that the quadratic entropy dissipation law is satisfied. Again based on this formulation, a third order  DG scheme
was further developed in \cite{LY14b} to numerically preserve the maximum principle:  if $c_1 \leq u_0(x)/M \leq c_2$, then $c_1 \leq u(x, t)/M \leq c_2$ for all $t>0$.     However, the non-logarithmic Landau formulation does not apply directly to the more general class of equations (\ref{fp}a).

{In this work, we construct an arbitrary high order entropy satisfying DG scheme for solving (\ref{fp}). The main idea behind the scheme construction is to apply the DG discretization to the following reformulation
\begin{equation}\label{uq+}
\partial_t u=\partial_x (f(u)\partial_x q), \quad  q =\Phi(x)+H'(u),
\end{equation}
by using a special numerical flux for  $\partial_x q$.
The resulting scheme is shown to feature several nice properties:   (i)  the entropy dissipation law (\ref{en}) is satisfied at
 the discrete level; (ii) the steady states are shown to be preserved; (iii) for the third order scheme applied to the model with a trivial potential, a sufficient condition on the range of flux parameters is rigorously  established so that cell averages remain positive at each time step, as long as each cell polynomial is positive at three test points.
For the numerical positivity a reconstruction algorithm based on positive cell averages is introduced so that the positivity of cell polynomials is enforced, without destroying the accuracy, at least for smooth solutions. This reconstruction also serves as a limiter imposed upon the numerical solution to suppress spurious oscillations at the solution singularity near zero.  For the general case the positivity of cell averages can be achieved by carefully tuning the parameters in the numerical flux, as illustrated
in the numerical experiments.

The discontinuous Galerkin (DG) method we discuss in this paper is a class of finite element methods, using a completely discontinuous piecewise polynomial space for the numerical solution and the test functions. One main advantage of the DG method was the flexibility afforded by local approximation spaces combined with the suitable design of numerical fluxes crossing cell interfaces. More general information about DG methods for elliptic, parabolic, and hyperbolic PDEs can be found in the recent books and lecture notes \cite{HW07, Li06,  Ri08, Sh09}. Following the methodology of the direct discontinuous Galerkin (DDG) method proposed in \cite{LY09, LY10},  we adopt a similar numerical flux formula for $\partial_x q$ in (\ref{uq+}).  The main feature in the DDG schemes proposed in \cite{LY09, LY10} lies in numerical flux choices for the solution gradient, which involve higher order derivatives evaluated crossing cell interfaces.


The plan of the paper is as follows. In Section 2, we present our DG scheme in one dimensional setting.  In Section 3 we prove  several important  properties of the scheme, including the semi-discrete entropy dissipation law in Theorem 3.1, the fully-discrete entropy dissipation law in Theorem 3.3, 
the preservation of positive cell averages for the model with trivial potential in Theorem 3.4, and the preservation of steady states in Theorem 3.5. 
In Section 4, we elaborate various details in numerical implementation, including the reconstruction algorithm,  the time discretization, and the spatial  
Numerical results  are in Section 5, where we verify experimentally the high order spatial accuracy of our scheme and simulate the long-time behavior of numerical solutions. The proposed scheme is applied to several physical models including the porous medium equation, the nonlinear diffusion with a double-well potential, and the general Fokker--Planck equation. The numerical results confirm both the high order of accuracy and the  numerical efficiency to capture the large-time asymptotic. Concluding remarks are given in Section 6.

\section{DG discretization in space}	
In this section,  we present our DG scheme for (\ref{fp}).  For clarity of presentaiton , we restrict ourselves to the problem
 in one spatial dimension. It is straightforward to generalize this construction
 for Cartesian meshes in multidimensional case.

In  one-dimensional setting, let $\Omega = [a, b]$ be a bounded interval.  We divide $\Omega$ with  a mesh
$$
a=x_{1/2}<x_1<\cdots <x_{N-1/2}<x_N<x_{N+1/2}=b,
$$
and the mesh size $\Delta x_j=x_{j+1/2}-x_{j-1/2}$, and a family of $N$ control cells
$I_j =(x_{j-1/2}, x_{j+1/2})$  with cell center $x_j =(x_{j-1/2} +x_{j+1/2})/2$.  We denote by $v^+$ and $v^-$ the right and left limits of function $v$, and define
$$
[v]=v^+-v^-, \quad \{v\}=\frac{v^++v^-}{2}.
$$
Define an $k-$degree discontinuous finite element space
$$
V_h= \left\{v\in L^2(\Omega), \quad v|_{I_j}\in P^k(I_j), j\in \mathbb{Z}_N \right \},
$$
where $P^k(I_j)$ denotes the set of all polynomials of degree at most $k$ on $I_j$, and $\mathbb{Z}_r=\{1, \cdots, r\}$ for any positive integer $r$.

We rewrite the equation \eqref{fp} as follows
\begin{subequations}\label{fp+}
\begin{align}
& \partial_t u=\partial_x (f(u)\partial_x q), \\
& q =\Phi(x)+H'(u).
\end{align}
\end{subequations}
The DG scheme is to find  $(u_h, q_h)\in V_h \times V_h$ such that for all $v, r \in V_h$ and $j\in \mathbb{Z}_N$,
\begin{subequations}\label{DG}
\begin{align}
& \int_{I_j} \partial_t u_{h} vdx =-\int_{I_j} f(u_h)\partial_xq_{h} \partial_x vdx +\{f(u_h)\}\widehat{\partial_xq_{h}}v|_{\partial I_j} +\{f(u_h)\}\partial_xv (q_h-\{q_h\})|_{\partial I_j},\\
 & \int_{I_j}q_hrdx =\int_{I_j}(\Phi(x) +H'(u_h))rdx.
\end{align}
\end{subequations}
Here
$$
v|_{\partial I_j}=v(x_{j+1/2}^-) - v(x_{j-1/2}^+),
$$
and $\widehat{\partial_x q_{h}}$ is the numerical flux, following \cite{LY10}, taken as
\begin{equation}
\widehat{\partial_xq_{h}}=\beta_0 \frac{[q_h]}{h} +\{\partial_xq_{h}\}+\beta_1h[\partial_x^2q_{h}], \label{flux}
\end{equation}
where $h=\Delta x$ for uniform meshes and $h=(\Delta x_j +\Delta x_{j+1})/2$ at $x_{j+1/2}$ for non-uniform meshes.  Here $\beta_i, i=0, 1$  are parameters  satisfying  a condition of the form
$$
\beta_0>\Gamma(\beta_1),
$$
where $\Gamma(\beta_1)$ is chosen to ensure certain stability property of the underlying PDE.  

Note that if zero-flux boundary conditions of the form $\partial_x (\Phi(x)+H'(u))=0$ are specified, we simply set $q$-related terms on the domain boundary
to be zero. If a Dirichlet boundary condition  for $u$ is given at $\partial \Omega$, we define the boundary numerical flux \eqref{flux}
in the following way: 
\begin{subequations}\label{IC}
\begin{align}
\{f(u_h)\}&= \frac{f(u(a, t)) +f(u_h^+)}{2} \text{  if  }{x=a}; \quad  \frac{f(u_h^-)+f(u(b, t))}{2} \text{  if  }{x=b}, \\
[q_h] & =
\left\{\begin{array}{ll}
q_h^+ - \left( \Phi(a) +H'(u(a, t)) \right) &\quad \text{ for } x=a, \\
 \left( \Phi(b) +H'(u(b, t))  \right)- q_h^- &\quad \text{ for } x=b,
\end{array}
\right. \\
\{\partial_xq_{h}\}&=\partial_xq_{h}^+ \text{  if  }{x=a}; \quad  \partial_xq_{h}^- \text{  if  }{x=b}, \\
[\partial_x^2q_{h}]&=0.
\end{align}
\end{subequations}
Here the boundary conditions are built into the scheme in such a way that the boundary data are used when available, otherwise the value of the numerical solution in corresponding end cells will be used.

\section{Properties of the DG scheme} In this section, we investigate several desired properties of the semi-discrete DG scheme (\ref{DG}), and its time discretization.
\subsection{Entropy dissipation}  We first state the entropy satisfying property of DG scheme (\ref{DG}), using the following notation:
\begin{equation}\label{qe}
\|q_h\|_E^2:= \left[ \sum_{j=1}^N\int_{I_j} f(u_h) |\partial_xq_h|^2dx +\sum_{j=1}^{N-1}\left. \{f(u_h)\}
\left(  \frac{\beta_0}{h}[q_h]^2 \right)\right|_{x_{j+\frac12}} \right].
\end{equation}
\begin{thm} Consider the DG scheme \eqref{DG}-\eqref{flux}, subject to zero-flux
boundary condition. If $f(u_h)\geq 0$, then the semi-discrete entropy
$$
E(t)=\sum_{j=1}^N\int_{I_j} (\Phi u_h+H(u_h))dx
$$
satisfies
\begin{equation}\label{ed}
\frac{d}{dt} E(t)\leq -\gamma \|q_h\|^2_E
\end{equation}
for $\gamma=1-\sqrt{\frac{\Gamma}{\beta_0}} \in (0, 1)$, provided
\begin{equation}\label{b0}
\beta_0>\Gamma (\beta_1):= \max_{1\leq j\leq N-1} \frac{\{f(u_h)\}
 \left( \{\partial_xq_{h}\} +\frac{\beta_1}{2} h[\partial_x^2q_{h}]\right)^2\Big|_{x_{j+1/2}} }{\frac{1}{2h}\left(\int_{I_j}+\int_{I_{j+1}}\right)f(u_h)|\partial_xq_{h}|^2dx}.
\end{equation}
\end{thm}

\begin{proof}
Summing \eqref{DG}-\eqref{flux}  over all index $j$ we obtain a global formulation:
\begin{align}\label{utv}
\int_\Omega \partial_tu_{h}v dx & = - \sum_{j=1}^N \int_{I_j} f(u_h) \partial_xq_{h} \partial_xvdx - \sum_{j=1}^{N-1} \{f(u_h)\} \left(
\widehat{\partial_xq_{h}} [v]+ \{ \partial_xv\}[q_h]\right)_{j+1/2},\\ \label{qr}
\int_{\Omega} q_h r dx  & =\int_\Omega (\Phi +H'(u_h))rdx.
\end{align}
Take $r=\partial_t u_{h}$  in (\ref{qr}) to obtain
$$
\int_\Omega \partial_tu_{h} q_hdx =\int_\Omega ( \Phi(x)+H'(u_h))\partial_t u_{h}dx=\frac{d}{dt}\int_\Omega (\Phi u_h+H(u_h))dx
=\frac{d}{dt}E(t).
$$
The right hand side from taking $v=q_h$ in (\ref{utv}) becomes
\begin{align*}
\frac{d}{dt}E(t) &= -\sum_{j=1}^N\int_{I_j} f(u_h)|\partial_xq_{h}|^2dx -\sum_{j=1}^{N-1} \{f(u_h)\} \left( \widehat{\partial_xq_{h}}[q_h]
+\{\partial_xq_{h}\}[q_h]\right)_{j+1/2}\\
 & =   -\sum_{j=1}^N\int_{I_j} f(u_h)|\partial_xq_{h}|^2dx -\sum_{j=1}^{N-1} \{f(u_h)\} \left(  \beta_0 [q_h]^2/h +[q_h]( 2\{\partial_xq_{h}\}
 +\beta_1h[\partial_x^2q_{h}])
 \right)_{j+1/2}.
\end{align*}
Using Young's  inequality we obtain
$$
- (2\{\partial_xq_{h}\} + \beta_1h[\partial^2_xq_{h}]) [q_h]\leq \beta_0(1-\gamma)[q_h]^2/h+ \frac{h}{4\beta_0(1-\gamma)} \left( 2\{\partial_xq_{h}\}
+\beta_1h[\partial_x^2q_{h}]\right)^2
$$
for some $0<\gamma<1$.  Hence
\begin{align}
\frac{d}{dt}E(t) & \leq  -\gamma\left[ \sum_{j=1}^N\int_{I_j} f(u_h)|\partial_xq_{h}|^2dx +\sum_{j=1}^{N-1}\left(  \frac{\{f(u_h)\}\beta_0}{h}[q_h]^2 \right)_{j+1/2} \right] \label{DE1}\\ \notag
& \quad - \left[ (1-\gamma)\sum_{j=1}^N\int_{I_j} f(u_h)|\partial_xq_{h}|^2dx -\sum_{j=1}^{N-1} \frac{h\{f(u_h)\}}{4\beta_0(1-\gamma)}  \left( 2\{\partial_xq_{h}\} +\beta_1h[\partial_x^2q_{h}]\right)^2 \right]\\ \notag
& \leq  -\gamma\left[ \sum_{j=1}^N\int_{I_j} f(u_h)|\partial_xq_{h}|^2dx +\sum_{j=1}^{N-1}\left(  \frac{\{f(u_h)\}\beta_0}{h}[q_h]^2 \right)_{j+1/2} \right]  \\ \notag
& \quad -\frac{1-\gamma}{2}\int_{I_1\cup I_N} f(u_h)|\partial_xq_{h}|^2dx,  \label{DE3}
\end{align}
since $\beta_0$ satisfies (\ref{b0}), hence
$$
\beta_0(1-\gamma)^2 = \Gamma \geq \frac{\sum_{j=1}^{N-1} h\{f(u_h)\} \left(\{\partial_xq_{h}\} +\frac{\beta_1}{2}h[\partial_x^2q_{h}]\right)_{j+1/2}^2}
{ \left( \sum_{j=2}^{N-1}\int_{I_j} + \frac{1}{2}\int_{I_1\cup I_N}\right) f(u_h)|\partial_xq_{h}|^2dx}.
$$
This finishes the proof of (\ref{ed}).
\end{proof}
\begin{rem} We remark that  a larger,  yet simpler,  $\Gamma(\beta_1)$ can be found for sufficiently small $h$ since the variation of ratio $\frac{\{f\}}{f}$ is also small. Assume that this ratio is bounded by a factor  $2$,  i.e., $ 2 \geq \frac{f}{\{f\}}\geq \frac{1}{2}$, then
\begin{align*}
 \Gamma(\beta_1) &  \leq 2 \max_{1\leq j\leq N-1}  \frac{\left( \{\partial_xq_{h}\} +\frac{\beta_1}{2} h[\partial_x^2q_{h}]\right)^2\Big|_{x_{j+1/2}} }
{\frac{1}{2h}\left(\int_{I_j}+\int_{I_{j+1}}\right)|\partial_xq_{h}|^2dx} \\
&\leq 2 \max_{1\leq j\leq N-1}\frac{\left(\frac{\partial_xq_h^--\beta_1h\partial^2_xq_h^-}{2}\right)_{x_{j+1/2}}^2+ \left(\frac{\partial_xq_h^++\beta_1h\partial^2_xq_h^+}{2}\right)^2_{x_{j+1/2}}}
{\frac{1}{2h}\left(\int_{I_j}|\partial_xq_{h}|^2dx +\int_{I_{j+1}}|\partial_xq_{h}|^2dx\right)}
\end{align*}
It is clear that this inequality is implied by
\begin{equation}
 \Gamma(\beta_1) \leq2 \max_{1\leq j\leq N-1}\left\{\frac{(\partial_xq_h^--\beta_1h\partial_x^2q_h^-)^2}{\frac{1}{2h}\int_{I_j}|\partial_xq_h|^2},\frac{(\partial_xq_h^++\beta_1h\partial_x^2q_h^+)^2}{\frac{1}{2h}\int_{I_{j+1}}|\partial_xq_h|^2}\right\}.
\label{gamma_inter}
\end{equation}
By setting $v(\xi)=\partial_xq_h\left(x_j+\frac{h}{2}\xi\right)$ for $q_h(x)|_{I_j}$, and $v(\xi)=\partial_xq_h\left(x_{j+1}-\frac{h}{2}\xi\right)$ for $q_h|_{I_{j+1}}$, we have
$$
 \Gamma(\beta_1) \leq 2 \sup_{v\in P^{k-1}} \frac{(v(1) -2\beta_1 \partial_\xi v(1))^2}{ \frac{1}{2}\int_{-1}^1 |v|^2d\xi}
 =2k^2 \left(1-\beta_1 (k^2-1)+\frac{\beta_1^2}{3}(k^2-1)^2 \right),
$$
here we have used the exact formula in \cite[Lemma 3.1]{Liu14}. Hence it suffices to choose $\beta_0$ such that
\begin{equation}\label{b0k}
\beta_0>2k^2 \left(1-\beta_1 (k^2-1)+\frac{\beta_1^2}{3}(k^2-1)^2 \right).
\end{equation}

 \end{rem}
\begin{rem} The positivity of numerical solutions are realized through a reconstruction algorithm at each time step, based on positive cell averages,
as detailed in Section 4.1.  It is shown in Theorem \ref{thk2}  that the use of non-zero $\beta_1$ is crucial in the sense that the positivity of cell averages
can be ensured. Indeed, this is proved for the third order DG scheme in solving (\ref{fp}) with zero potential. For the model with non-trivial potential, our numerical experiments again confirm the special role of $\beta_1$ in the preservation of  positivity  of  numerical cell averages.
\end{rem}

\subsection{The fully-discrete DG scheme}
In order to preserve the entropy dissipation law for $u_h^n$ at each time step, the time step restriction  is needed when using an explicit time discretization.
We now discuss this issue by taking the Euler first order time discretization of (\ref{DG}): find  $u_h^{n+1}(x)\in V_h$ such that for any $r(x), v(x) \in V_h$,
\begin{subequations}\label{fully_DG}
\begin{align}
&\int_{I_j}q_h^n r\,dx =\int_{I_j}\left(\Phi(x) +H'(u_h^n) \right) r\,dx,\\
&\int_{I_j}D_t u_h^n v \,dx =-\int_{I_j} f(u_h^n) \partial_x q_{h}^n \partial_x v\,dx
	+\{f(u_h^n)\}\left.\left[\widehat{\partial_x q_h^n }v + \partial_x v (q_h^n-\{q_h^n\})\right]\right|_{\partial I_j}.
\end{align}
\end{subequations}
Here and in what follows, we use the notation for any function $w^n(x)$ as
$$
D_t w^n =\frac{w^{n+1}-w^n}{\Delta t},
$$
and $\mu = \frac{\Delta t}{h^2}$ as the mesh ratio.
\begin{lem} The following  inverse inequalities hold for any $v\in V_h$:
\begin{subequations}\label{in}
\begin{align}
& \sum_{j=1}^N \int_{I_j} v_x^2 dx \leq \frac{k(k+1)^2(k+2)}{h^2} \sum_{j=1}^N  \int_{I_j}v^2 dx, \\
& \sum_{j=1}^{N-1} [v]_{j+1/2}  \leq \frac{4(k+1)^2}{h} \sum_{j=1}^N  \int_{I_j}v^2 dx, \\
& \sum_{j=1}^{N-1} \{v_x\}^2_{j+1/2} \leq \frac{k^3(k+1)^2(k+2)}{h^2} \sum_{j=1}^N  \int_{I_j}v^2 dx.
\end{align}
\end{subequations}
\end{lem}
\begin{proof} These follow from the repeated use of the two inverse inequalities:
\begin{subequations}
\begin{align}
& \max\{|w(a)|, |w(b)|\} \leq (m+1)|I|^{-1/2}\|w\|_{L^2(I)},\\
& \|\partial_x w\|_{L^2(I)}\leq (m+1) \sqrt{m(m+2)}|I|^{-1}\|w\|_{L^2(I)},
\end{align}
\end{subequations}
provided $w \in P^m(I)$  with $I=(a, b)$ and $|I|=b-a$. The first bound is well known, see e.g. \cite{WH03}.
The second inequality may be found in \cite[Lemma 3.1]{LP15}
\end{proof}

\begin{thm}\label{th4.2}
Let the fully discrete entropy be defined as
$$
E^n =\sum_{j=1}^N\int_{I_j} \left(\Phi(x) u_h^n(x)+H(u_h^n(x))\right)dx.
$$
The DG scheme \eqref{fully_DG},  subject to zero-flux boundary condition, satisfies
 \begin{align}\label{en+}
D_t E^n\leq -\frac{\gamma}{2} \|q_h^n\|_E^2
\end{align}
for some $\gamma \in (0, 1)$, provided $u_h^n(x)$ remains positive, $\beta_0>\Gamma(\beta_1)$, and
\begin{align}\label{cfl+}
 \mu \leq \frac{\gamma }{C(k, \beta_0, \beta_1)\|\max\{0, H''(u_h^n(\cdot))\}\|_\infty\|f(u_h^n(\cdot))\|_\infty},
 \end{align}
 where $C(k, \beta_0, \beta_1)$ is given in (\ref{ck}) below.
\end{thm}
\begin{proof}
Summing \eqref{fully_DG} over all index $j$'s we obtain
\begin{align}
\sum_{j=1}^N\int_{I_j} q_h^n r \,dx
& = \sum_{j=1}^N\int_{I_j}\left(\Phi(x) +H'(u_h^n) \right) r\,dx,\label{q_sum}\\
\sum_{j=1}^N\int_{I_j} D_t u_h^n  v\, dx & = - \sum_{j=1}^N \int_{I_j} f(u_h^n)  \partial_xq_h^n \partial_x v\,dx
	- \sum_{j=1}^{N-1} \{f(u_h^n) \} \left.\left(\widehat{\partial_xq^n_h} [v]+ \{ \partial_x v\}[q^n_h]\right)\right|_{x_{j+\frac12}}. \label{rho_sum}
\end{align}
Take $ r =D_t u_h^n $  in (\ref{q_sum}) to obtain
\begin{align*}
 \int_\Omega D_t u_h^n q_h^n dx  & =\int_\Omega \left( \Phi(x)+H'(u_h^n(x))  \right)D_t u_h^n\,dx \\
 & =D_t E^n -\frac{1}{\Delta t} \int_{\Omega} ( H(u^{n+1}_h) - H(u^n_h) - H'(u_h^n) (u^{n+1}_h -u_h^n))dx \\
& = D_t E^n -\frac{\Delta t}{2} \int_{\Omega}H''(\cdot)(D_t u_h^n)^2dx.
\end{align*}
Here $(\cdot)$ denotes the intermediate value between $u_h^n$ and $u_h^{n+1}$. Taking $v =q_h^n$, (\ref{rho_sum}) becomes
\begin{align*}
 \int_\Omega D_t u_h^n q_h^n dx  &= -\sum_{j=1}^N\int_{I_j} f(u_h^n) |\partial_xq_h^n |^2\,dx -\sum_{j=1}^{N-1} \{f(u^n_h) \} [q_h^n]\left.\left( \widehat{\partial_xq_h^n}+\{\partial_xq_h^n\}\right)\right|_{x_{j+\frac12}}\\
& \leq  -\gamma \|q_h^n\|_E^2,
\end{align*}
for $\beta_0$ satisfying (\ref{b0}) at each interface $x_{j+\frac12}$, $j=1,\ldots, N-1.$ Hence
$$
D_t E^n\leq -\gamma \|q_h^n\|_E^2 + \frac{\Delta t}{2} \int_{\Omega}  H''(\cdot) (D_t u_h^n)^2\,dx.
$$
The claimed estimate follows if
\begin{equation}\label{cfl}
\Delta t \leq \frac{\gamma \|q_h^n\|_E^2}{\int_{\Omega} \max\{0, H''(\cdot)\} (D_t u_h^n)^2\,dx}.
\end{equation}
For convex $H$, this indeed imposes a time restriction.

It remains to show that the bound in (\ref{cfl+}) is smaller than the right side of (\ref{cfl}).   In \eqref{rho_sum},  we take $v=D_tu^n_h$ and use the Young inequality $ab \leq \frac{1}{4\epsilon}a^2+\epsilon b^2$ to obtain
\begin{align*}
\sum_{j=1}^N\int_{I_j} v^2\, dx & = - \sum_{j=1}^N \int_{I_j} f(u_h^n)  \partial_xq_h^n \partial_x v\,dx
	- \sum_{j=1}^{N-1} \{f(u_h^n) \} \left.\left(\widehat{\partial_xq^n_h} [v]+ \{ \partial_x v\}[q^n_h]\right)\right|_{x_{j+\frac12}}\\
&\leq \frac{1}{4\epsilon_1h^2} \sum_{j=1}^N \int_{I_j} f^2(u_h^n)  |\partial_xq_h^n|^2 \,dx +\epsilon_1 h^2 \sum_{j=1}^N \int_{I_j} |\partial_x v| ^2\,dx\\
   & +  \frac{1}{4\epsilon_2h} \sum_{j=1}^{N-1} \left.\{f(u_h^n) \}^2 |\widehat{\partial_xq^n_h}|^2 \right|_{x_{j+\frac12}} +\epsilon_2h \sum_{j=1}^{N-1}\left. { [v]^2}\right|_{x_{j+\frac12}} \\
   & + \frac{1}{4\epsilon_3h^3} \sum_{j=1}^{N-1} \left.\{f(u_h^n) \}^2 [q^n_h]^2\right|_{x_{j+\frac12}}  +\epsilon_3h^3 \sum_{j=1}^{N-1} \left.\{\partial_xv\}^2\right|_{x_{j+\frac12}} .
\end{align*}
The use of inequalities in (\ref{in}) leads to
\begin{align*}
& \epsilon_1  h^2 \sum_{j=1}^N \int_{I_j} |\partial_x v| ^2\,dx
    + \epsilon_2 h  \sum_{j=1}^{N-1} \left.{ [v]^2}\right|_{x_{j+\frac12}}
    + \epsilon_3h^3  \sum_{j=1}^{N-1}\left. [\partial_xv]^2 \right|_{x_{j+\frac12}} \\
  &    \leq(k+1)^2   ( k(k+2) \epsilon_1 + 4\epsilon_2+ k^3(k+2) \epsilon_3)
    \sum_{j=1}^N\int_{I_j} v^2\, dx \\
    & = \frac{3}{4}  \sum_{j=1}^N\int_{I_j} v^2\, dx,
 \end{align*}
 provided
 $$
 (4\epsilon_1)^{-1} =k(k+1)^2(k+2),  \; (4\epsilon_2)^{-1} =4 (k+1)^2, \quad (4\epsilon_3)^{-1} =k^3(k+1)^2(k+2).
 $$
 This gives
\begin{align}\label{inv}
& \frac{1}{4}\sum_{j=1}^N\int_{I_j} v^2\, dx \leq \frac{k(k+1)^2(k+2)}{h^2}\sum_{j=1}^N \int_{I_j} f^2(u_h^n)  |\partial_xq_h^n|^2 \,dx \\ \notag
                                           & \qquad+\frac{k^3(k+1)^2(k+2)}{h^3} \sum_{j=1}^{N-1} \left.\{f(u_h^n) \}^2 [q^n_h]^2\right|_{x_{j+\frac12}}+\frac{4(k+1)^2}{h}\sum_{j=1}^{N-1} \left.\{f(u_h^n) \}^2 |\widehat{\partial_xq^n_h}|^2 \right|_{x_{j+\frac12}}.
\end{align}
It is clear that the first two terms are bounded by $\|f(u_h^n(\cdot)\|_\infty \|q_h^n\|_E^2$.  We now show that the last term is also bounded by $\|f(u_h^n(\cdot)\|_\infty \|q_h^n\|_E^2$, up to constant multiplication factors.
\begin{align*}
\sum_{j=1}^{N-1} \left. \{f(u_h^n) \} |\widehat{\partial_xq^n_h}|^2\right|_{x_{j+\frac12}} &
=
 \sum_{j=1}^{N-1} \left.\{f(u_h^n) \} \left|\{\partial_xq_h^n\}+\beta_0\frac{[q^n_h]}{h}+\beta_1h[\partial_x^2q^n_h]\right|^2
\right|_{x_{j+\frac12}}\\
&\leq  2 \sum_{j=1}^{N-1} \left. \{f(u_h^n) \} \left(\beta_0^2\frac{[q^n_h]^2}{h^2}+\left(\{\partial_xq^n_h\}+\beta_1h[\partial_x^2q^n_h]\right)^2
\right)\right|_{x_{j+\frac12}}.
\end{align*}
From (\ref{b0}) it follows that
$$
 \left.\{f(u_h^n) \}\left(\{\partial_xq^n_h\}+\beta_1h[\partial_x^2q^n_h]\right)^2\right|_{x_{j+\frac12}}  \leq \frac{ \Gamma(2\beta_1) }{2h}
\left(\int_{I_j}+\int_{I_{j+1}}\right)f(u_h)|\partial_xq_{h}|^2dx.
$$
Hence
$$
 \sum_{j=1}^{N-1} \left.\{f(u_h^n) \}\left(\{\partial_xq^n_h\}+\beta_1h[\partial_x^2q^n_h]\right)^2\right|_{x_{j+\frac12}} \leq \frac{ \Gamma(2\beta_1)}{h}
\sum_{j=1}^N\int_{I_j}f(u_h)|\partial_xq_{h}|^2dx.
$$
These together yield
\begin{align*}
 \sum_{j=1}^{N-1} \left. \{f(u_h^n) \} |\widehat{\partial_xq^n_h}|^2\right|_{x_{j+\frac12}}  \leq \frac{2}{h} \max\{\beta_0, \Gamma(2\beta_1)\} \|q^n_h\|_E^2.
\end{align*}
Upon insertion into (\ref{inv}) we obtain
\begin{align*}
\sum_{j=1}^N\int_{I_j} v^2\, dx \leq  \frac{C(k, \beta_0, \beta_1) ||f(u_h^n(\cdot))||_\infty}{h^2}\|q^n_h\|_E^2,
\end{align*}
where
\begin{equation}\label{ck}
C(k, \beta_0, \beta_1):= 4(k+1)^2 \left(k(k+2)\max\{1, k^2/\beta_0\}+ 8\max\{\beta_0, \Gamma(2\beta_1)\} \right).
\end{equation}
Hence (\ref{cfl}) is implied by (\ref{cfl+}).

This ends the proof.
\end{proof}
\subsection{Preservation of positive cell averages}  It is known to be difficult, if not impossible, to preserve point-wise solution bounds for high order numerical approximations.  A popular strategy after the work  \cite{ZS10} is to combine an accuracy preserving reconstruction with the bound preserving property of cell averages.
For the DG scheme applied to  (\ref{fp}) with $\Phi=0$, following \cite{LY14b},  we are able to identify a range of $\beta_1$ so that positive averages are ensured for at least the third order scheme.   We have not been able to prove this property for the general case.

 By taking the test function $v = 1$ on $I_j$ in  (\ref{fully_DG}b), we obtain the evolutionary equation for the cell average,
\begin{equation}\label{ca}
 \bar u^{n+1}_j=  \bar u^n_j  + \mu h\left. \{f(u_h^n)\} \widehat{\partial_x q_h^n}\right|_{\partial I_j}.
\end{equation}
 {For the case that $H$ is convex and  $\Phi(x)=0$,  we reformulate (\ref{fp+}) as
\begin{equation*}
\partial_t u=\partial_x (fH''\partial_x q), \quad  q=u.
\end{equation*}
At the discrete level, we simply set $q_h=u_h$ and replace $f$ by $fH''$ in  (\ref{fully_DG}b).}
Assuming that $ \bar u^n_j \in [c_1, c_2]$ for all $j$'s, we can derive some sufficient conditions such that  $ \bar u^{n+1}_j \in [c_1, c_2] $ under certain CFL condition on $\mu$.

For piecewise quadratic polynomials, we have the following result.
\begin{thm}\label{thk2}($k=2$) The scheme (\ref{ca}) with $q_h=u_h$,  and
\begin{align}\label{betak2}
\frac{1}{8}< \beta_1 < \frac{1}{4} \quad \text{and} \quad  \beta_0 \geq 1
\end{align}
is bound preserving, namely, $\bar{u}_{j}^{n+1}\in [c_1, c_2]$  if $u_h^n(x) \in [c_1, c_2]$ on the set $S_j$'s where
$$
S_j = x_j +\frac{h}{2}\left\{-1, 0, 1\right\},
$$
under the CFL condition
\begin{align}\label{cf}
\mu \leq \mu_0= \frac{1}{12\max_{1\leq j\leq N}  |f(u_{j-1/2}^n)|} \min \left\{
	\frac{1}{\beta_0 +8\beta_1-2 },
	\frac{1}{ 1-4\beta_1}
	\right\}.
\end{align}
\end{thm}
\begin{proof}  Let
$$
p(\xi)=u_h\left(x_j+\frac{h}{2}\xi \right) \text{ for } \xi\in [-1,1],
\quad \text{i.e.}, \quad p =u_h|_{I_j},
$$
we have
\begin{equation}\label{avek2}
\bar u_j =  \frac{1}{6} p(-1)+   \frac{2}{3}p(0)+  \frac{1}{6} p(1).
\end{equation}
In what follows we denote $p_-=u_h|_{I_{j-1}}$ and $p_+=u_h|_{I_{j+1}}$.

We  represent the diffusion flux in terms of solution values over the set $S_j$; see \cite{LY14b}.
\begin{align}\label{fluxk}
h \left.\widehat{\partial_xu_h}\right|_{x_{j+\frac12}}
= \alpha_3 p_{+}(-1)+\alpha_2 p_{+}(0)+\alpha_1p_{+}(1)- \left(\alpha_1 p(-1)+\alpha_2 p(0)+\alpha_3p(1)\right),
\end{align}
where
\begin{align}\label{alphas}
\alpha_1= \frac{8\beta_1-1}{2}, \quad
\alpha_2 =2(1-4\beta_1), \quad
\alpha_3=\beta_0 +\frac{8\beta_1-3}{2}.
\end{align}
It is easy to verify that (\ref{betak2}) ensures $\alpha_i \geq 0$ for $i=1, 2, 3$.

Upon substitution into (\ref{ca}) we obtain
\begin{align}\label{avedeck2}
\bar u_j^{n+1} = & \bar u_j + 2 \mu \left(\left. h \{f(u_h)\} \widehat{\partial_xu_h}\right|_{x_{j+\frac12}}
	- \left. h \{f(u_h)\} \widehat{\partial_x u_h}\right|_{x_{j-\frac12}}\right) \\
=&\left[  \frac{1}{6} - 2\mu \left(\alpha_3 f_{j-\frac12} +\alpha_1 f_{j+\frac12} \right) \right]p(-1) \nonumber\\
&+\left[  \frac{2}{3} - 2\mu \left(\alpha_2 f_{j-\frac12} +\alpha_2 f_{j+\frac12}\right)\right]p(0)\nonumber\\
&+\left[  \frac{1}{6} - 2\mu \left(\alpha_1 f_{j-\frac12} +\alpha_3 f_{j+\frac12} \right) \right]p(1) \nonumber\\
&+2\mu f_{j+\frac12}\left[\alpha_3p_{+}(-1)+\alpha_2p_{+}(0)+\alpha_1p_{+}(1) \right]\nonumber \\
&+2\mu f_{j-\frac12}\left[\alpha_1 p_-(-1)+\alpha_2p_-(0)+\alpha_3p_-(1) \right].\nonumber
\end{align}
Here we have used the notation
$$
f_{j+\frac12}:= \left.\{f(u_h)\}\right|_{x_{j+\frac12}} = \left.\frac{f(u_h^-)+f(u_h^+)}{2}\right|_{x_{j+\frac12}}.
$$
Note that the sum of all coefficients of above polynomial values is one.  Hence $\bar u^{n+1}_j\in [c_1, c_2]$ as long as
$u_h^n\in [c_1, c_2]$ on $S_j$ and all coefficients are nonnegative. The nonnegativity imposes a CFL condition
$\mu \leq \mu_0$ with $\mu_0$ being
\begin{align*}
\frac{1}{12} \min_{1\leq j\leq N}\left\{
	\frac{1}{\alpha_3 f_{j-\frac12} + \alpha_1 f_{j+\frac12}},
	\frac{4}{\alpha_2 f_{j-\frac12} + \alpha_2 f_{j+\frac12}},
	\frac{1}{\alpha_1 f_{j-\frac12} + \alpha_3 f_{j+\frac12}}
\right\}.
\end{align*}
Here we assume that $f_{N+1/2}=0$ so that $j=N$ can be included in the above expression. It suffices to take  smaller
$$
\mu_0= \frac{1}{12 \max_{ }|f(u^n_{j-1/2})|} \min\left\{
	\frac{1}{\alpha_3 + \alpha_1},
	\frac{2}{\alpha_2}
\right\}.
$$
That is  (\ref{cf}), as claimed.
\end{proof}

\begin{rem}The CFL condition (\ref{cf}) is  sufficient conditions rather than necessary to preserve the bound of solutions.
Therefore, in practice, these CFL conditions are strictly enforced only in the case the bound preserving property is violated.
\end{rem}
\begin{rem} For general case, we expect there is  still a proper set of parameters $(\beta_0,\beta_1)$ with which the scheme can  preserve
positivity of cell averages. Our numerical simulations in Example 2 confirms this expectation.
\end{rem}

\subsection{Preservation of steady states}
If we start with an initial data $u_h^0$, already at steady states, i.e., $\Phi(x) +H'(u_h^0(x))=C$,  it follows from (\ref{fully_DG}a)
that $q_h^0=C$.  Furthermore,   (\ref{fully_DG}b) implies that $u_h^1=u_h^0\in V_h$. By induction we have
$$
 \Phi(x) +H'(u_h^n(x))=C\quad \forall n\in \mathbb{N}.
$$
This says that the DG scheme (\ref{fully_DG}a) preserves the steady states.    Moreover, we can show that in some cases  the numerical solution tends
asymptotically toward a steady state, independent of initial data.  More precisely, we have the following result.
\begin{thm}  Let the assumptions in Theorem \ref{th4.2} be met,  and $(u_h^n, q_h^n) $ be the numerical solution to the fully discrete DG scheme (\ref{fully_DG}),   then the limits of $(u_h^n, q_h^n)$ as  $n\to \infty$ satisfy
$$
q^*_h = C, \quad \Phi(x) +H'(u^*_h) \in C+ V_h^\bot,
$$
where $C$ is a constant. For quadratic $H(u)$,  $C$ can be determined  explicitly by
$$
C=\frac{1}{|\Omega|}{\int_{\Omega} (\Phi(x) + H'(u_0)(x))dx}.
$$
In addition, if $\Phi(x)\in P^m (m \leq k)$, then we must have $\Phi(x) +H'(u^*_h(x)) \equiv C$.
\end{thm}
\begin{proof}
Since $E^n$ is non-increasing and bounded from below,  we have
$$
\lim_{n \to \infty} E^n=\inf\{E^n\}.
$$
Observe from (\ref{en+}) that
$$
E^{n+1} -E^{n} \leq -\frac{\gamma \Delta t}{2} \|q_h^n\|_E^2 \leq 0.
$$
When passing the limit $n\to \infty$ we have $\lim_{n\to \infty} \|q_h^n\|_E^2=0$.  This implies that
each term in this energy norm must have zero as its limit, that is
\begin{equation}\label{limitn}
\lim_{n\to \infty} \sum_{j=1}^N \int_{I_j} f(u_h^n) |\partial_x q_h^n|^2dx=0, \quad \lim_{n\to \infty} \sum_{j=1}^{N-1} \frac{\beta_0}{h}\{f(u_h^n)\}[q_h^n]^2
\Big|_{j+\frac12}=0.
\end{equation}
The first relation in (\ref{limitn}) tells that  the limit of $q_h^n$, denoted by  $q_h^*$, must be constant in each computational cell.
The second relation in (\ref{limitn}) infers that $q_h^*$  must be a constant in the whole domain.  These when inserted into (\ref{fully_DG}a) gives the
desired result.  For quadratic $H(u)$, we use the mass conservation $\int_{\Omega} H'(u^*_h(x))dx=\int_{\Omega} H'(u_0(x))dx$ to determine the constant $C$.
The proof is complete.
\end{proof}
\begin{rem}The above result shows that for quadratic $H(u)$ and potential $\Phi(x)$ being polynomials of degree up
to $k$, the steady states are approached by numerical solutions.  For other cases, such asymptotic convergence holds only in the projection sense.
\end{rem}

\section{Numerical implementation}
In this section, we provide further details in implementing the entropy satisfying discontinuous Galerkin (ESDG) method.

\subsection{Reconstruction} For a high order polynomial approximation, numerical solutions can have negative values.
We enforce the solution positivity through some accuracy-preserving reconstruction. Motivated by the definite result on the bound preserving property of cell averages for special cases in Theorem \ref{thk2},  we  consider the case with positive cell averages.

Let $w_h \in  P^k(I_j)$  be an approximation to a smooth function $w(x) \geq 0$, with cell averages $\bar{w}_j>\delta$
for $\delta$ being some small positive parameter or zero.  We then reconstruct  another polynomial in $P^k(I_j)$
so that
\begin{equation}\label{ureconstruct}
\tilde w_h^{\delta}(x)= \bar{w}_j+\frac{\bar{w}_j-\delta}{\bar{w}_j-\min_{I_j} w_h(x)} (w_h(x)-\bar{w}_j),\quad
 \text{ if } \min_{I_j} w_h(x)<\delta.
\end{equation}
This reconstruction maintains same cell averages and satisfies  $$\min_{I_j} w^\delta(x)\geq\delta.$$
It is known that enforcing a maximum principle numerically might damp oscillations in numerical solutions, see, e.g.
\cite{LO96, ZS10}.
Numerical example in Fig.\ref{fig:porousm2exact} confirms such a damping effect  near zero from using the positivity
preserving limiter (\ref{ureconstruct}).

\begin{lem} If $\bar{w}_j>\delta$,  then the reconstruction satisfies the estimate
$$|w^{\delta}(x)-w_h(x)|\leq C(k) \left( ||w_h(x)-w(x)||_\infty+ \delta\right),\quad \forall x\in I_j,$$
where $C(k)$ is a constant depending on $k$.  This says that the reconstructed $w^{\delta}(x,t)$ in \eqref{ureconstruct}
does not destroy the accuracy when $\delta<h^{k+1}$.
\end{lem}
\begin{proof}
We have
\begin{align*}
 |w^\delta(x)-w_h(x)|&
                            = \left|   \frac{\delta- \min_{I_j} w_h(x)}{\bar{w}_j-\min_{I_j} w_h(x)} (\bar{w}_j-w_h(x)) \right|\\
                            &\leq  \frac{\max_{I_j} |\bar{w}_j -w_h(x)|}{\max_{I_j} (\bar{w}_j-w_h(x))} \left( ||w_h(x)-w(x)||_\infty+ \delta\right).
\end{align*}
It follows from \cite{LY14b, ZS10} that
  $$\frac{\max_{I_j} |\bar{w}_j-w_h(x)|}{\max_{I_j} (\bar{w}_j-w_h(x))} \leq C(k),$$
  where $k$ is the degree of the polynomial $w_h(x)$.
\end{proof}

\subsection{Time discretization} For the time discretization of (\ref{DG}), we use the explicit high order Runge-Kutta method. The explicit time discretization is simple to implement, with entropy dissipation law still preserved under some restriction on the time step.

 Let $\{t^n\}, n=0, 1,\ldots$ be a uniform partition of time interval.
 Denote $u_h^n \sim u(t_n, x)$, $q_h^n \sim q(t_n, x)$, where $t_n=n\Delta t$ and $\Delta t$ is the uniform temporal step size. The algorithm can be summarized in following steps.
\begin{itemize}
\item[1.] Project $u_0(x)$ onto $V_h$ to obtain $u_h(0)$ and solve (\ref{DG}b)  to obtain $q_h(0)$.
\item[2.] Solve (\ref{DG}a) to obtain $u_h^{n+1}$ with a Runge-Kutta (RK) ODE solver. Perform reconstruction \eqref{ureconstruct} if needed.
\item[3.] Solve (\ref{DG}b) to obtain $q_h^{n+1}$ from the obtained $u_h^{n+1}$.
\item[4.] Repeat steps 2 and 3 until final time $T$.
\end{itemize}

In our numerical simulation we choose  $\Delta t=C(k)h^2$,  where $C(k)$ is smaller for larger $k$.   For the case with zero potential and $k=2$, $C(k)$ is given in Theorem \ref{thk2}.   The choice of the time step $\Delta t \sim h^2$ suggests that we adopt an $m^{th}$ order Runge-Kutta solver with $m\geq(k+1)/2$, so that in the accuracy test  the temporal error is smaller than the spatial error. For polynomials of degree $k=1,2,3$, we use the second order explicit Runge-Kutta method (also called Heun's method) to solve the ODE system $\dot a=\mathfrak{L}(\textbf{a})$:
\begin{align}
	{\textbf{a}^{(1)}} &= \textbf{a}^n + \Delta t \mathfrak{L}(\textbf{a}^n), \nonumber \\
	\textbf{a}^{n+1} &=   \frac{1}{2}\textbf{a}^n + \frac{1}{2}\textbf{a}^{(1)}+ \frac{1}{2} \Delta t\mathfrak{L}( {\textbf{a}^{(1)}}). \nonumber
\end{align}
The bound preserving property for cell averages in Theorem 3.3,  depending on a convex combination of polynomial values in previous time step, works well with the above Runge-Kutta solver since it is simply a convex combination of the forward Euler.

\subsection{Spatial discretization} In this section, we present some further details on the spatial discretization.
The $k$th order basis functions in a 1-D standard reference element $\xi \in [-1, 1]$ are taken as
the Legendre polynomials $\{L_i(\xi)\}_{i=0}^k$,  then the numerical solutions in each cell $x\in I_j$ can be expressed as
\begin{align*}
 u_h(x, t) =\sum_{i=0}^ku_j^i(t) L_i(\xi)=: L^\top(\xi)u_j(t), \quad q_h(x, t)  =\sum_{i=0}^kq_j^i(t)L_i(\xi)=: L^\top(\xi)q_j(t),
\end{align*}
using a uniform mesh size $h$ and the map $x=x_j+\frac{h}{2}\xi$, with notation $L^\top=(L_0, L_1, \cdots, L_k)$ and $u_j=(u_j^0, \cdots, u_j^k)^\top$.

For given $\Phi(x)$, a simple calculation of (\ref{DG}a) with $v=L(\xi)$ gives
\begin{equation}\label{3}
M \dot u_j=\frac{2}{h} R_1+\frac{1}{2h}( R_2+R_3),  \quad 2\leq j \leq N-1,
\end{equation}
where
\begin{align*}
M& = \frac{h}{2}\int_{-1}^1 L(\xi)L^\top(\xi)d\xi,\\
R_1 & =- \sum_{i=1}^{Q}  \omega_i f\left( L^\top(s_i) u_j(t) \right)  L_\xi^\top(s_i)q_j L_\xi(s_i),\\
R_2 & =\left(f\left(L^\top(1)u_j\right)+f\left(L^\top(-1)u_{j+1}\right)\right)( - D^\top  q_j+ E^\top q_{j+1})L(1) \\
& \qquad - \left( f\left( L^\top(1)u_{j-1}\right)+f\left(L^\top(-1)u_j\right) \right)( - D^\top q_{j-1}+ E^\top q_{j})L(-1)=R_2^+-R_2^-, \\
R_3 & =\left( f\left(L^\top(1)u_j\right)+ f\left(L^\top(-1)u_{j+1}\right)\right) ( L^\top(1) q_j- L^\top(-1)q_{j+1})L_\xi(1)\\
&\qquad   + \left( f\left(L^\top(1)u_{j-1}\right)+ f\left(L^\top(-1)u_{j}\right)\right)(L^\top(1) q_{j-1}- L^\top(-1)q_{j})L_\xi(-1)=: R_3^+ + R_3^-.
\end{align*}
Here
$$
D=\beta_0L(1)-L_\xi(1)+4\beta_1L_{\xi\xi}(1), \quad  E =\beta_0L(-1)+L_\xi(-1)+4\beta_1L_{\xi\xi}(-1).
$$
In the evaluation of $R_1$,  we choose $Q$ Gaussian quadrature points $s_i \in[-1, 1]$ with $1\leq i\leq Q$. Here and in what follows, we choose $Q$ quadrature points with $Q\geq \frac{k+2}{2}$ so that the quadrature rule with accuracy of order $\mathcal{O}(h^{2Q})$ does not destroy the scheme accuracy.  At two end cells, if the zero flux conditions are specified, we use  $R_2=R_2^+, R_3=R_3^+$ for $j=1$ and $R_2=-R_2^-, R_3= R_3^-$ for $j=N$.

If Dirichlet boundary conditions, $u(a)$ and $u(b)$,  are specified,
we modify $R_2$ and $R_3$ according to \eqref{IC}. That is, for $j=1$,
\begin{align*}
 R_2 & =R_2^+ -  (f(u(a)) +f\left(L^\top(-1)u_{1}\right))[\beta_0(L^\top(-1)q_1-\Phi(a)-H'(u(a))) + 2L_\xi^\top(-1)q_1 ]L(-1), \\
R_3 & =R_3^+ + (f(u(a)) +f\left(L^\top(-1)u_{1}\right)) [ \Phi(a)+H'(u(a))-L^\top(-1)q_1] L_\xi(-1),
\end{align*}
and  for $j=N$,
\begin{align*}
R_2 & =(f\left(L^\top(1)u_N\right) +f(u(b))) [-\beta_0(L^\top(1)q_N-\Phi(b)-H'(u(b))) + 2L_\xi^\top(1)q_N ]L(1)-R_2^-, \\
R_3 & = f\left(L^\top(1)u_N\right) +f(u(b))) [ L^\top(1)q_N-\Phi(b)-H'(u(b)) ] L_\xi(1)+ R_3^-.
\end{align*}

To solve  (\ref{DG}b) is, using the $Q$-point Gauss quadrature rule on the interval $(-1, 1)$,  to solve
\begin{equation}\label{3+}
M q_j=\frac{h}{2} \sum_{i=1}^{Q} \omega_i (\Phi(x(s_i))+H'(L^\top(s_i)u_j))L(s_i).
\end{equation}
The collection of (\ref{3}) and (\ref{3+}) with $1\leq j\leq N$ forms a nonlinear ODE system, for which we use a Runge-Kutta method.

\section{Numerical Tests}
In this section, we present  a selected set of numerical examples in order to numerically validate our ESDG scheme.  Via several physical models from different applications, we examine the order of accuracy by  numerical convergence tests, while we quantify $l_1$ errors defined by
$$
\|u_h-u_{ref}\|_{l_1}= \sum_{j=1}^N \int_{I_j}|u_h(x) - u_{ref}(x)|dx,
$$
with the integral on $I_j$ evaluated by a $4$-point Gaussian quadrature method and $u_{ref}$ being a reference solution obtained by using a refined mesh size.
It is also demonstrated that the scheme captures well the long-time behavior of underlying solutions, as well as the mass concentration phenomenon in certain applications.

\subsection{Porous medium equation} We consider the porous medium equation of the form
\begin{equation}\label{pm}
\partial_t u =\partial_x^2 (u^m), \quad m>1.
\end{equation}
With this model we will illustrate 1) the scheme's capability in capturing the solution singularity; 2) the positivity preservation proved in Theorem \ref{thk2}.\\

\noindent{\bf Example 1. Capturing singularity}\\
Barenblatt and Pattle independently found an explicit solution of \eqref{pm} when the {Dirac delta function} is used as initial condition {\cite{Barenblatt,Pattle}}.  A special explicit solution which we will use is
\begin{equation}
B_m(x, t)=\max \left \{0, t^{-\alpha}\left(
0.2-\frac{\alpha(m-1)}{2m}\frac{|x|^2}{t^{2\alpha}}
\right)^{\frac{1}{m-1}}\right\},\quad \alpha=\frac{1}{m+1}.
\end{equation}
We compute the solution of \eqref{pm} with initial data $u_0(x)=B_2(x, 0.1)$, with zero flux boundary conditions $\partial_x u(\pm 2, t)=0$.

Fig.\ref{fig:porousm2exact} shows the exact solution and $P^2$ numerical solutions without and with reconstruction {\eqref{ureconstruct} with $\delta$ set to be $0$. This reconstruction is not applied  to the cells where the $u_h$ are entirely zero. The scheme with reconstruction gives sharp resolution of expanding fronts, keeping the solution strictly within the initial bounds.  The scheme without reconstruction brings visible undershoots near the foot of the numerical solution.

{Fig.\ref{fig:porousm2compare} shows a numerical comparison for polynomials with different degrees,  $k=1,2,3$. Cell averages are  shown in Fig.\ref{fig:porousm2compare} (left) and cell polynomials in Fig.\ref{fig:porousm2compare}(right) (zoomed  near singularity), we can clearly see that
a higher order method gives  a  more accurate approximation.} \\

\begin{figure}[!htb]
\caption{Capturing singularity in the exact solution at $t=0.5$}
\centering
\includegraphics[scale=.7]{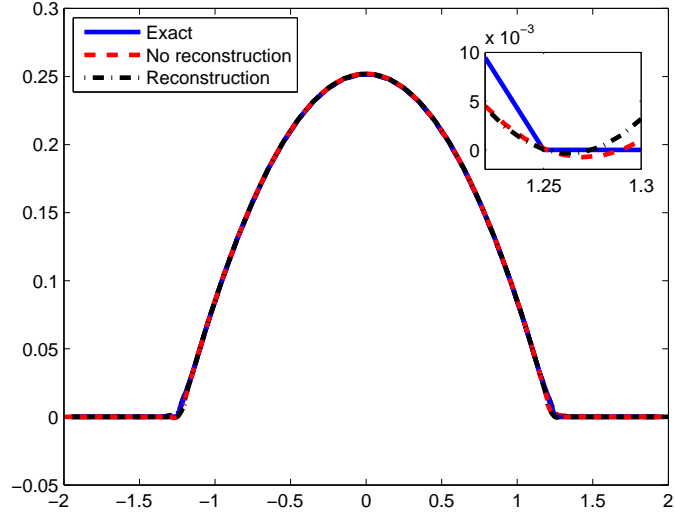}
\label{fig:porousm2exact}
\end{figure}

\begin{figure}[!htb]
\caption{Comparison of solutions for $k=1,2,3$}
\centering
\begin{tabular}{cc}
\includegraphics[scale=0.42]{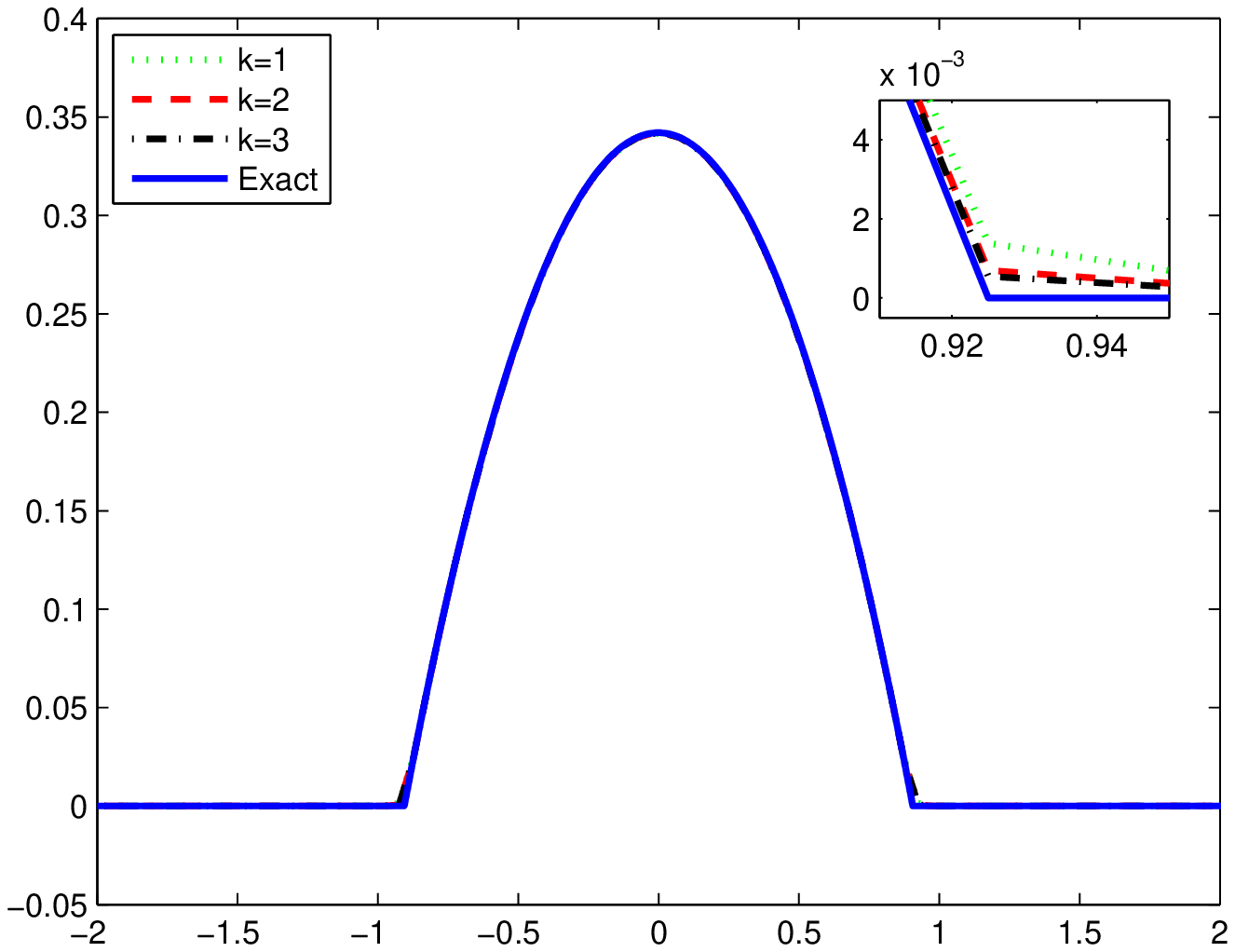}&
 \includegraphics[scale=0.42]{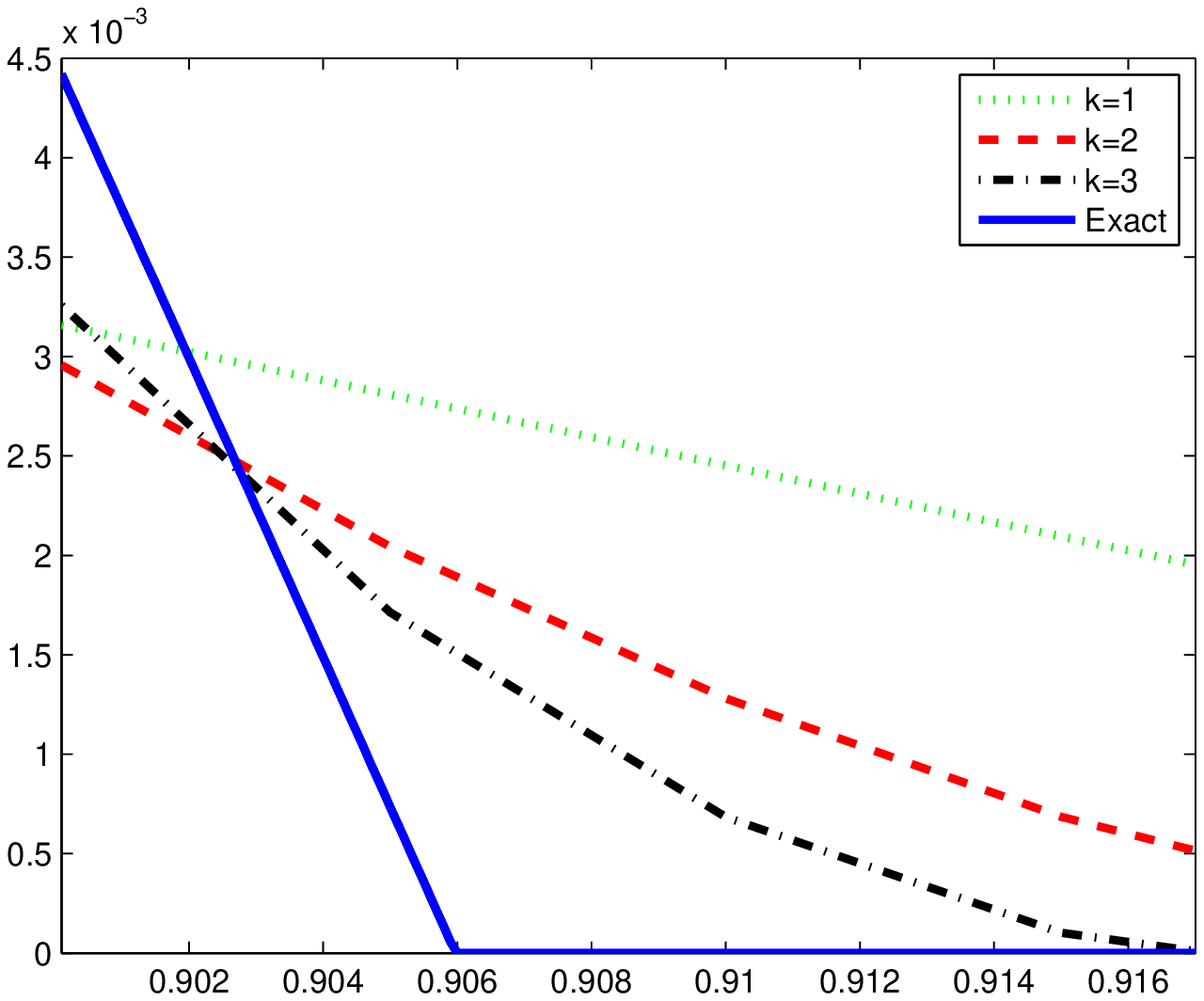}\\
 cell averages & cell polynomials
 \end{tabular}
\label{fig:porousm2compare}
\end{figure}

\noindent{\bf Example 2. Positivity preservation}\\
In this example we test the effect of using different parameter $\beta_1$ in terms of the positivity preservation.  Equation \eqref{pm} with $m=2$,  when written in the form
$$\partial_tu=\partial_x(f(u)\partial_xq),\quad f(u)=2u,\quad q=u,$$
satisfies the requirements  in    Theorem \ref{thk2}.  We consider {positive} initial data with small amplitude,
$$u_0(x)=\epsilon(1+30e^{-25x^2}),\quad x\in [-1,1],$$
 and zero flux boundary conditions $\partial_x u(\pm 1, t)=0$.   With $\epsilon=10^{-5}$, $\delta=10^{-10}$, $h=0.2$, $k=2$ and $\Delta t=0.25h^2$ in the simulation, our results indicate that cell average $\bar u$ remains above $\delta$  at $t=1000$ when using  $(\beta_0, \beta_1)=(2, 1/6)$; while $\bar u$ already becomes negative at $t=41.388$ when taking $(\beta_0, \beta_1)=(2, 0)$. This is consistent with the conclusion in Theorem \ref{thk2} that $ \beta_1\in (1/8, 1/4)$ is sufficient for positivity preservation of cell averages, and for any other $\beta_1$'s such a property is not guaranteed. {We note here that the range of $\beta_1$ in Theorem \ref{thk2} is only sufficient. Our simulation also indicates that cell average $\bar u$ still remains above $\delta$  at $t=1000$ when using  $(\beta_0, \beta_1)=(2, 1/2)$, which does not satisfy the requirement in Theorem \ref{thk2}.}

We further test the special effect of parameter $\beta_1$ on the positivity preservation for the case with nontrivial potential,
$\Phi=30\epsilon x^2/2$, i.e., we have
$$
\partial_t u = \partial_x(f(u)\partial_xq),\quad f(u)=2u,\quad q=u+30\epsilon x^2/2.
$$
Though Theorem \ref{thk2} is no longer applicable due to the nonzero potential,  we still see similar effects of $\beta_1$  through numerical experiments. With the same initial condition and parameters as above, our simulation results in Table \ref{tab:beta1}  show that there is a range for $\beta_1$ in which $\bar u$ remains above $\delta$ at $t=1000$; while $\bar u$ becomes negative at $t<1000$ when $\beta_1\leq 1/6$ or $\beta_1\geq2$. This observation indicates that  1) $\beta_1$ plays a special role for the positivity preservation; 2) the admissibility of $\beta_1$ depends on the underlying problem.
\begin{table}[!htb]
\caption{Time when $\bar u$ becomes negative}
\begin{tabular}{ |c|c|c| }
\hline
 $(\beta_0, \beta_1)$ & negative $\bar u$ time  \\ \hline
 (2,0)      & 35.41 \\ \hline
 (2, 1/12)    & 388.91\\ \hline
 (2,1/6)    &   845.69\\ \hline
 (2,1/3)    & $>$1000 \\ \hline
  (2,1/2)    & $>$1000 \\ \hline
   (2,2/3)    & $>$1000 \\ \hline
  (2,1)    & $>$1000 \\ \hline
  (2,2)    & 917.42\\ \hline
   (2,3)    &  740.92\\ \hline
  \end{tabular}
\label{tab:beta1}
\end{table}

\subsection{Porous medium equation with linear convection}
We consider the following porous medium equation with linear convection
$$
\partial_ t u =\partial_x^2 (u^m) + \partial_x u, \quad m>1.
$$
This equation corresponds to (\ref{fp}a) with $f(u)=u$, $\Phi=x$ and $H=\frac{u^m}{m-1}$, and has a wide range of applications. With this model equation we shall test the numerical convergence and the scheme accuracy. We note that the case $m=2$ was tested in \cite{BF} with a second order finite volume scheme.
\\

\noindent{\bf Example 3 (m=2).} We consider $$ \partial_ t u =\partial_x^2(u^2)+\partial_xu,$$   with initial data
$$
u_0(x) = 0.5 + 0.5 \sin(\pi x),	\quad x \in [-1, 1],
$$
subject to  zero-flux boundary condition, that is $\partial_xu(\pm 1, t)=-\frac{1}{2}$. In Table 2 we observe that the orders of convergence  are of $\mathcal{O}(h^{k+1})$  for polynomials of  degree $k$ ($k=1,2,3$). \\
{\small
\begin{table}[!htb]
\caption{Error table  for porous media equation with $m=2$ at $t=1$}
\begin{tabular}{ |c|l|c|c| }

\hline
$(k,\beta_0, \beta_1)$& h & $l_1$ error & order \\ \hline
\multirow{4}{*}{$(1,1,-)$}
 & 0.4  &         0.0056949&       --       \\
 & 0.2  &         0.0013756&            2.15 \\
 & 0.1  &        0.00034588&           2.20  \\
 & 0.05 &       6.5394e-005&         2.40   \\ \hline
 \multirow{4}{*}{$(2,4,1/12)$}
  & 0.4  &      0.00026132 &       --       \\
  & 0.2  &     3.9026e-005 &            2.86   \\
  & 0.1  &     5.3072e-006 &            2.91     \\
\ & 0.05 &     6.8756e-007 &            2.95   \\ \hline

 \multirow{4}{*}{$(3,9,1/4)$}
  & 0.4  &   4.4584e-005 &       --       \\
  & 0.2  &   4.4365e-006 &              3.71 \\
  & 0.1  &   3.2099e-007 &              3.91  \\
\ & 0.05 &   1.9724e-008 &              4.02 \\ \hline

\end{tabular}
\label{tab:porousm2}
\end{table}
}

\noindent{\bf Example 4 (m=3). } We further test the case $m=3$, i.e.,
$$ \partial_ t u =\partial_x^2(u^3)+\partial_xu,$$
with initial data
$$
u_0(x) = 1+ 0.5 \sin(\pi x),	\quad x \in [-1, 1],
$$
subject to zero-flux boundary conditions $(uu_x)(\pm 1, t)=-1/3$. The numerical convergence test is performed with the same flux parameters
for each $k$ as in the previous example, both errors and orders of convergence are given in Table \ref{tab:porousm3}.  These  results further confirm the
$(k+1)$-th order of accuracy when using $P^{k} (k=1, 2, 3)$ elements.

\begin{table}[!htb]
\caption{Error table for porous medium equation with $m=3$ at $t=1$}
\begin{tabular}{ |c|l|c|c| }
\hline
$(k,\beta_0, \beta_1)$& h &  $l_1$  error & order \\ \hline
\multirow{4}{*}{$(1,1,-)$}
 & 0.4  &           0.0014749 &       --       \\
 & 0.2  &          0.00037363 &             1.99 \\
 & 0.1  &         9.5215e-005 &           1.99  \\
 & 0.05 &         2.3636e-005 &         2.01   \\ \hline
 \multirow{4}{*}{$(2,4,1/12 )$}
  & 0.4  &      7.3404e-005&       --       \\
  & 0.2  &      9.5432e-006&              2.97 \\
  & 0.1  &      1.2268e-006&            2.98 \\
\ & 0.05 &      1.5257e-007&             3.00 \\ \hline

 \multirow{4}{*}{$(3,9,1/4)$}
  & 0.4  &      5.1001e-006&       --       \\
  & 0.2  &      3.4917e-007&         3.96  \\
  & 0.1  &      2.1473e-008  &         4.00 \\
\ & 0.05 &      1.3609e-009&         3.98  \\ \hline
\end{tabular}
\label{tab:porousm3}
\end{table}
Numerical tests in Example 3 and 4 also indicate that cell averages can be made positive in time  when choosing proper
parameters $(\beta_0, \beta_1)$, together with reconstruction \eqref{ureconstruct} performed at each time step.

\subsection{Nonlinear diffusion with a double-well potential}
Consider a nonlinear diffusion equation with an external double-well potential of the form
$$
\partial_t u=\partial_x (u\partial_x  (\nu u^{m-1}+\Phi )),  \quad \Phi =\frac{x^4}{4}-\frac{x^2}{2}.$$
This model equation is taken from \cite{CCH}, and it corresponds to system \eqref{fp} with  $H'(u)=\nu u^{m-1}$.  With this model we shall test both numerical accuracy and the asymptotic behavior of numerical solutions.
 \\

\noindent{\bf Example 5. Free energy decay}\\
In this example, we take $\nu=1$, $m=2$ and initial data
 $$
 u_0(x)=\frac{0.1}{\sqrt{0.4\pi}}e^{-\frac{x^2}{0.4}}, \quad x\in [-2, 2],
 $$
subject to zero-flux boundary conditions $\partial_x u(\pm 2, t)=\mp 6$. Both errors and orders of convergence are given in Table  \ref{tab:diffusion}, {which again demonstrates    $\mathcal{O}(h^{k+1})$ order of accuracy for $P^k$ polynomials. }

\begin{table}
\caption{Error table for nonlinear diffusion with a double-well potential at $t=1$}
\begin{tabular}{ |c|l|c|c| }
\hline
$(k,\beta_0, \beta_1)$& h &$l_1$  error & order \\ \hline
\multirow{4}{*}{$(1,1,-)$}
 & 0.4  &     0.082882 &       --       \\
 & 0.2  &    0.0051793  &            2.70\\
 & 0.1  &    0.0012178  &             2.06 \\
 & 0.05 &  0.00029961 &            2.02\\ \hline
 \multirow{4}{*}{$(2,4,1/12)$}
  & 0.4  &      0.16726 &       --       \\
  & 0.2  &     0.020986 &         3.08 \\
  & 0.1  &    0.0023122&          3.18  \\
\ & 0.05 &   0.00027875&          3.05 \\ \hline
 \multirow{4}{*}{$(3,12,1/24)$}
  & 0.8  &        0.09677&       --       \\
  & 0.4  &       0.010059&                   3.82 \\
  & 0.2  &     0.00051784&                   4.10\\
\ & 0.1 &     3.4058e-005&                   3.93\\ \hline
\end{tabular}
\label{tab:diffusion}
\end{table}

 We also examine the decay of  the
 entropy
$$
E=\int_{-2}^2 \left(\Phi(x) u +H (u)\right) dx =\int^2_{-2} \left[\left(\frac{x^4}{4}-\frac{x^2}{2}\right)u+ \frac{u^2}{2}\right] dx.
$$
Figure \ref{fig:energydecay} (left) shows the semilog plot of the free energy decay until final time $T=40$, and Figure \ref{fig:energydecay} (right) displays  the snapshots of $u$ at different times, showing the time-asymptotic convergence of the numerical solutions towards the steady states.
\begin{figure}[!htb]
\caption{Entropy decay of nonlinear diffusion with double well potential}
\centering
\includegraphics[scale=0.51]{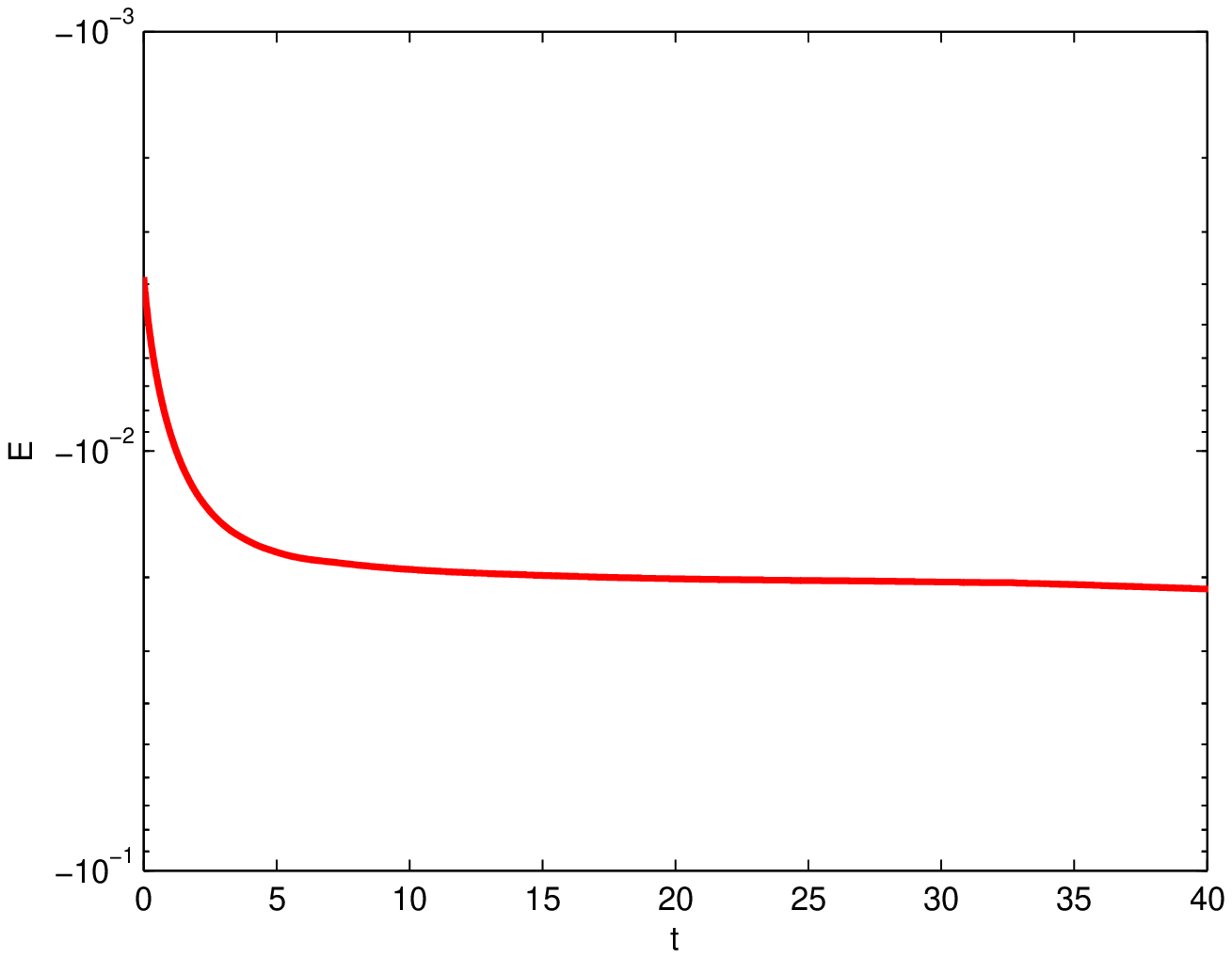}
 \includegraphics[scale=0.4]{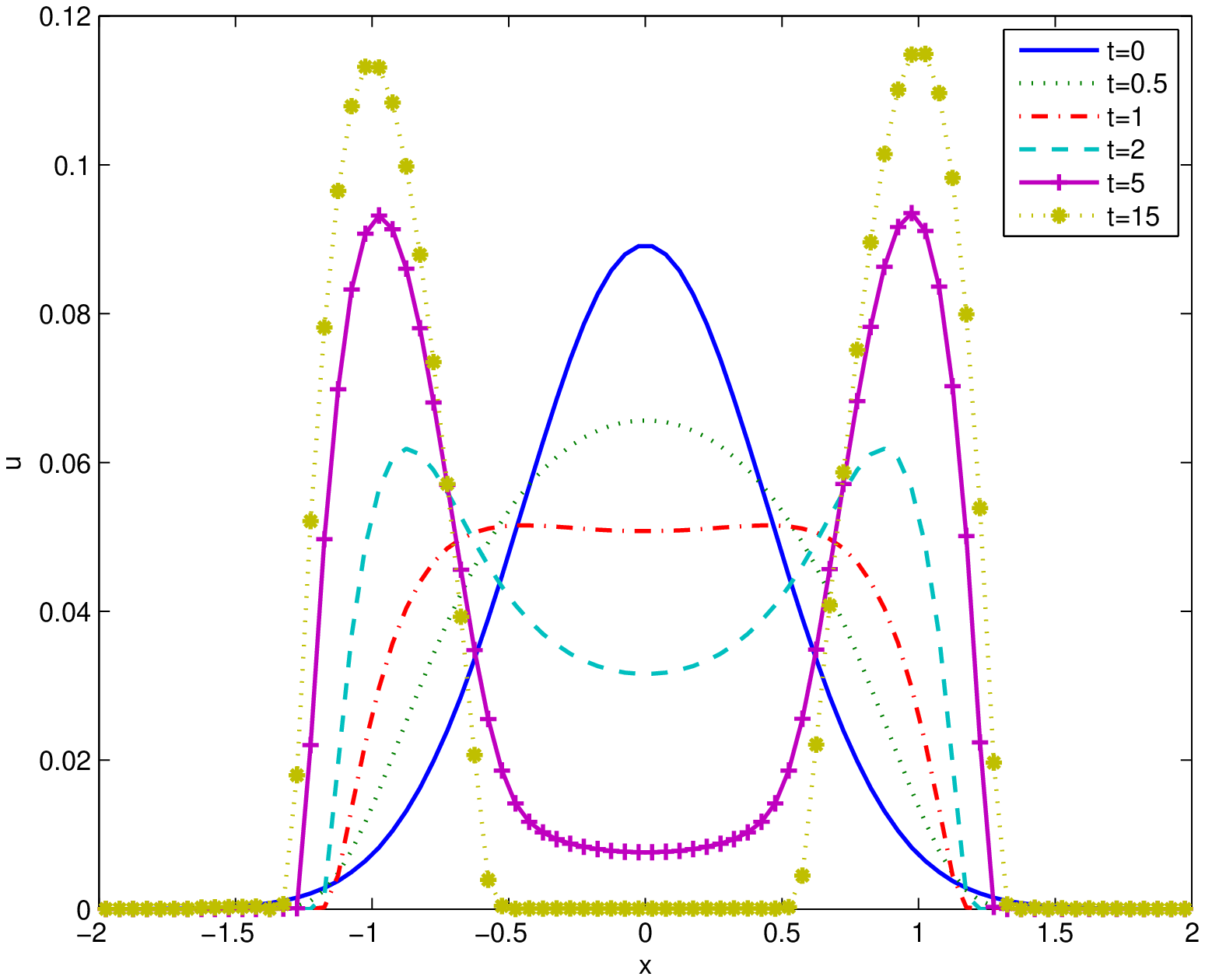}
\label{fig:energydecay}
\end{figure}

\bigskip

\subsection{The nonlinear Fokker-Planck equation} We consider the following model for boson gases,
\begin{align}\label{bg}
\partial_t u=\partial_x (x u(1+u^3)+\partial_x u),  \quad  t>0,
\end{align}
which is a nonlinear Fokker-Planck equation corresponding to (\ref{fp}a) with
$$
\Phi =\frac{x^2}{2},\quad f(u)=u(1+u^3), \quad H'(u)=\log \frac{u}{\sqrt[3]{1+u^3}}.
$$
This model equation exhibits  the critical mass phenomenon (see \cite{AGT}), that solutions with initial data of large mass blow-up in finite time, whereas solutions with initial data of small mass do not.  The authors in \cite{BF} numerically verified such critical mass phenomenon using a second order finite volume scheme. With our high order DG scheme, we test the critical mass phenomenon for (\ref{bg}) with initial data
$$
u_0(x)=\frac{M}{2\sqrt{2\pi}}\left( \exp\left(-\frac{(x-2)^2}{2}\right) + \exp\left(-\frac{(x+2)^2}{2}\right) \right),
$$
which has total mass $M$.  This is to illustrate the good performance of the ESDG scheme in capturing complex physical phenomena.
\bigskip

\noindent{\bf Example 6. Sub-critical mass $M=1$ and super-critical mass $M=10$}\\
We test the sub-critical mass $M=1$ with results in Figure \ref{fig:GFP} (left) and super-critical mass $M=10$ with results in Figure \ref{fig:GFP} (right) by $P^2$ polynomial approximations.
These results are consistent with the theoretical conclusion made in \cite{AGT} and the numerical observation in \cite{BF}, yet our scheme can produce numerical solutions with higher order of accuracy.  Note that  the reconstruction \eqref{ureconstruct} has to be implemented due to the involvement of $\log$-function in $H'(u)$.

\begin{figure}[!htb]
\caption{Dynamics of the general Fokker-Planck equation}
\centering
\begin{tabular}{cc}
\includegraphics[scale=.475]{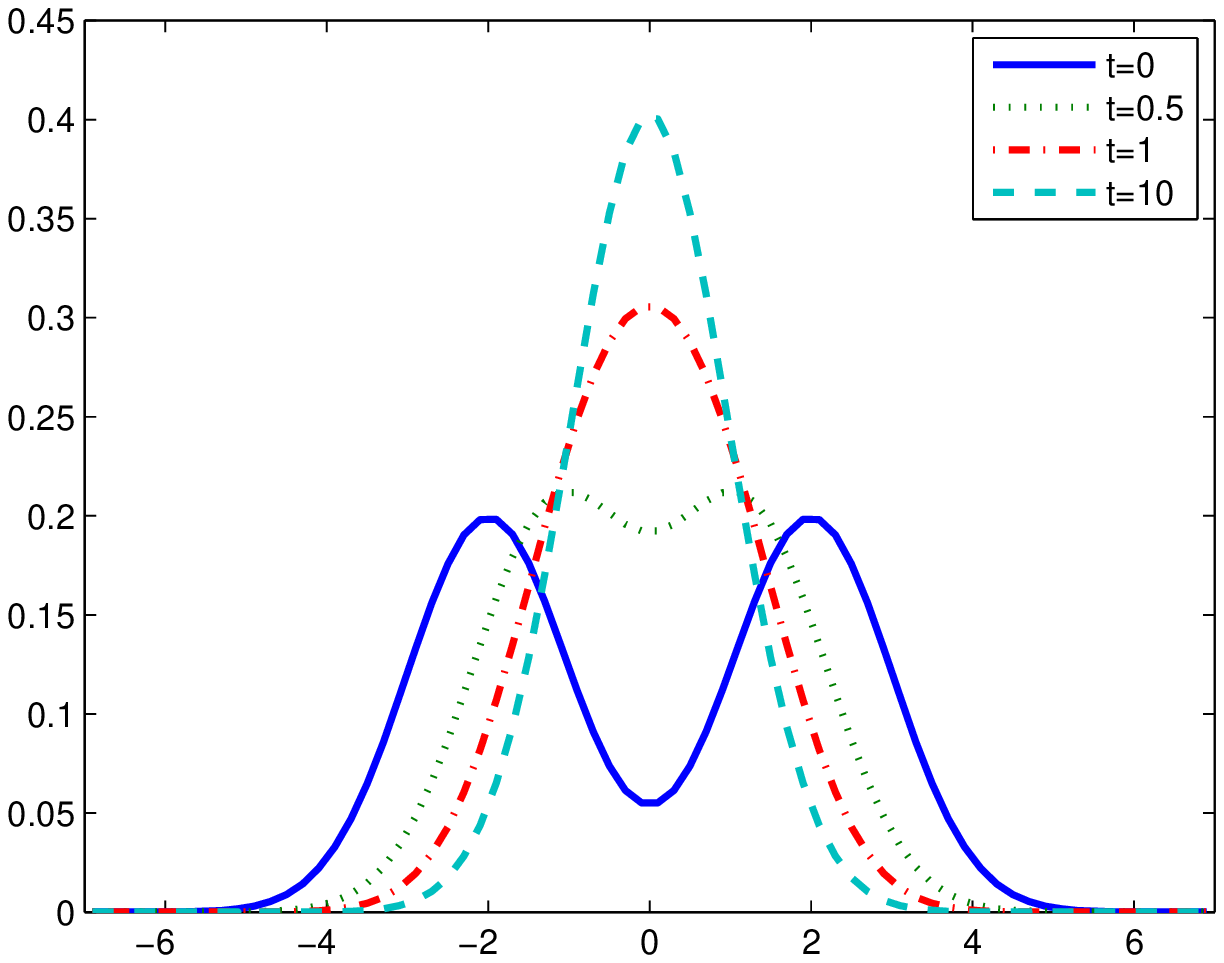}&
\includegraphics[scale=.5]{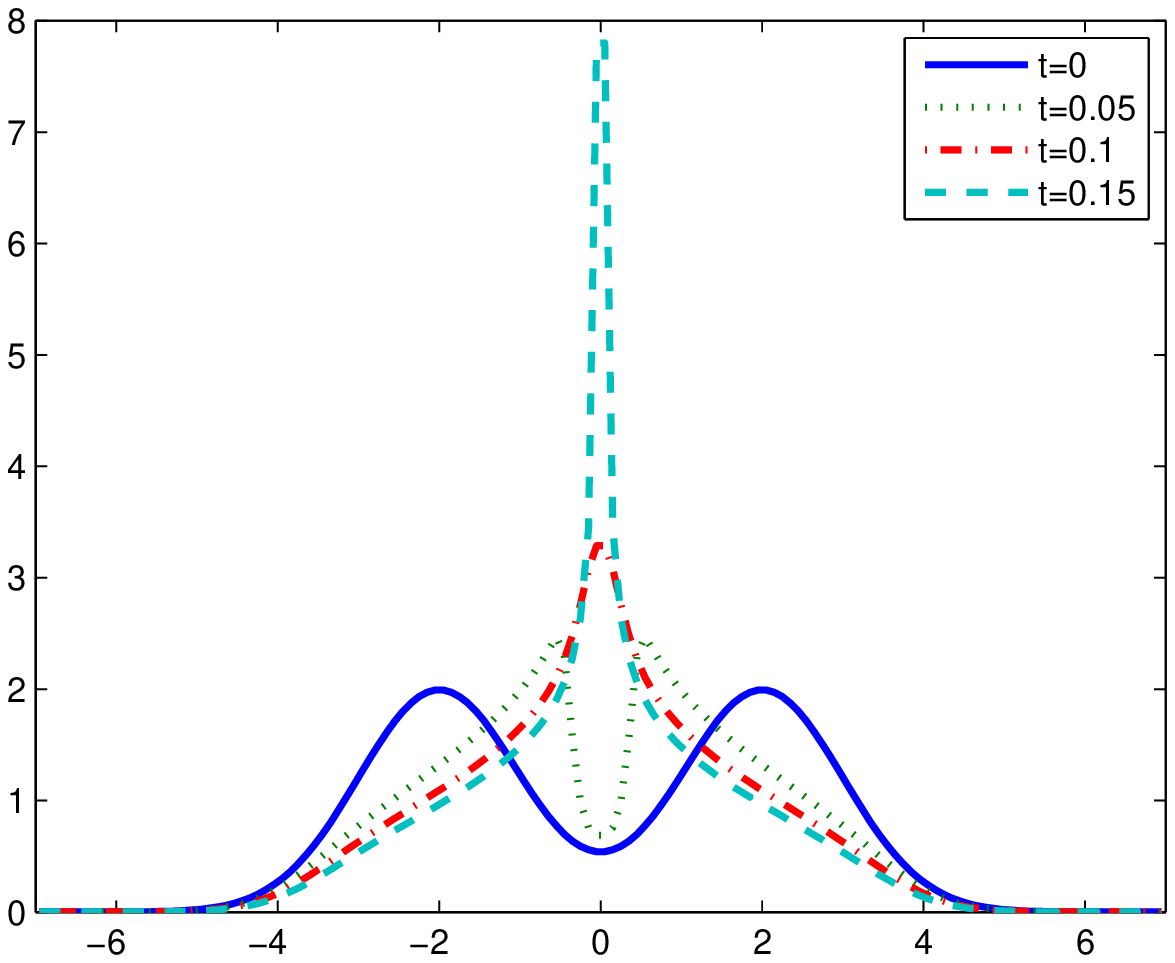} \\
{\small Sub-critical mass $M=1$}& {\small Super-critical mass $M=10$}
\end{tabular}
\label{fig:GFP}
\end{figure}

\section{Concluding remarks} In this article, we have developed an entropy satisfying DG method for solving nonlinear Fokker-Planck equations with a gradient flow structure. The idea is to rewrite the equation in the form of a convection equation with flux being $-f(u)\partial_xq$, and $q$ is obtained by a piecewise $L^2$ projection of $\Phi(x) +H'(u)$.  Then we apply the numerical flux of the DDG method introduced in \cite{LY10} to  $\partial_x q$.  The present scheme  is shown to satisfy a discrete version of the entropy dissipation law, therefore preserving steady-states and providing numerical solutions with satisfying long-time behavior. The positivity of numerical solutions is enforced through a reconstruction algorithm, based on positive cell averages. Cell averages can be made positive at each time step by carefully tuning the numerical flux parameter $(\beta_0, \beta_1)$.  For the model with trivial potential, a parameter range sufficient for positivity preservation is rigorously established. Numerical examples include the porous medium equation, the nonlinear diffusion equation with a double-well potential, and the general Fokker-Planck equation. Numerical results have demonstrated high-order accuracy of the scheme.  Moreover,  the long-time solution behavior is also examined to show the robustness of the proposed scheme.
\section*{Acknowledgments}
 Liu was supported by the National Science Foundation under Grant DMS1312636 and by NSF Grant RNMS
(Ki-Net) 1107291.


\end{document}